\documentclass[11pt]{amsart}%
\usepackage{graphicx,amscd,color,amsmath,amsfonts,amssymb,geometry}
\usepackage{amsmath}
\usepackage{amsfonts}
\usepackage{amssymb}
\usepackage{graphicx}%
\providecommand{\U}[1]{\protect\rule{.1in}{.1in}}
\newtheorem{theorem}{Theorem}[section]
\newtheorem{remark}[theorem]{Remark}
\newtheorem{lemma}[theorem]{Lemma}

\newtheorem{proposition}[theorem]{Proposition}

\begin{document}
\title[Remarks on the Hardy--Littlewood inequality ]{Remarks on the Hardy--Littlewood inequality for $m$-homogeneous polynomials
and $m$-linear forms}
\author[W. Cavalcante]{W. Cavalcante}
\address{Departamento de Matem\'{a}tica, \\
\indent Universidade Federal de Pernambuco, \\
\indent50.740-560 - Recife, Brazil.}
\email{wasthenny.wc@gmail.com}
\author[D. N\'{u}\~{n}ez]{D. N\'{u}\~{n}ez-Alarc\'{o}n}
\address{Departamento de Matem\'{a}tica, \\
\indent Universidade Federal de Pernambuco, \\
\indent50.740-560 - Recife, Brazil.}
\email{danielnunezal@gmail.com}
\author[D. Pellegrino]{D. Pellegrino}
\address{Departamento de Matem\'{a}tica, \\
\indent Universidade Federal da Para\'{\i}ba, \\
\indent 58.051-900 - Jo\~{a}o Pessoa, Brazil.}
\email{pellegrino@pq.cnpq.br and dmpellegrino@gmail.com}
\thanks{W. Cavalcante was supported by Capes D. N\'{u}\~{n}ez was supported by CNPq
Grant 461797/2014-3. D. Pellegrino was supported by CNPq Grant 477124/2012-7
and INCT-Matem\'{a}tica.}

\begin{abstract}
The Hardy--Littlewood inequality for $m$-homogeneous polynomials on $\ell_{p}$
spaces is valid for $p>m.$ In this note, among other results, we present an
optimal version of this inequality for the case $p=m.$ We also show that the
optimal constant, when restricted to the case of $2$-homogeneous polynomials
on $\ell_{2}(\mathbb{R}^{2})$ is precisely $2$. In an Appendix we justify why,
curiously, the optimal exponents of the Hardy--Littlewood inequality do not
behave smoothly.

\end{abstract}
\maketitle


\section{Introduction}

The Hardy--Littlewood inequality for (complex or real) bilinear forms defined
on $\ell_{p}$ spaces for $p>2$ dates back to 1934 \cite{hardy}. This
inequality together with the Bohnenblust--Hille inequality \cite{bh} and
Littlewood's 4/3 theorem \cite{LLL} are the cornerstones of the birth of the
fruitful theory of multiple summing operators. There are, of course, natural
counterparts of the Hardy--Littlewood inequality for $m$-homogeneous
polynomials and $m$-linear forms defined on $\ell_{p}$ spaces for $p>m$ (see
\cite{dimant} and the references therein).

For $\mathbb{K}$ be $\mathbb{R}$ or $\mathbb{C}$ and $\alpha=(\alpha
_{1},\ldots,\alpha_{n})\in{\mathbb{N}}^{n}$, we define $|\alpha|:=\alpha
_{1}+\cdots+\alpha_{n}$. By $\mathbf{x}^{\alpha}$ we shall denote the monomial
$x_{1}^{\alpha_{1}}\cdots x_{n}^{\alpha_{n}}$ for any $\mathbf{x}%
=(x_{1},\ldots,x_{n})\in{\mathbb{K}}^{n}$. The polynomial Littlewood's $4/3$
theorem asserts that there is a constant $B_{\mathbb{K},2}^{\mathrm{pol}}%
\geq1$ such that
\[
\left(  {\sum\limits_{\left\vert \alpha\right\vert =2}}\left\vert a_{\alpha
}\right\vert ^{\frac{4}{3}}\right)  ^{\frac{3}{4}}\leq B_{\mathbb{K}%
,2}^{\mathrm{pol}}\left\Vert P\right\Vert
\]
for all $2$-homogeneous polynomials $P:$ $\ell_{\infty}^{n}\rightarrow
\mathbb{K}$ given by
\[
P(x_{1},...,x_{n})=\sum_{|\alpha|=2}a_{\alpha}\mathbf{{x}^{\alpha},}%
\]
and all positive integers $n$, where $\Vert P\Vert:=\sup_{z\in B_{\ell
_{\infty}^{n}}}|P(z)|$. When we replace $\ell_{\infty}^{n}$ by $\ell_{p}^{n}$
we obtain the polynomial Hardy--Littlewood inequality whose optimal exponents
are $\frac{4p}{3p-4}$ for $4\leq p\leq\infty$ and $\frac{p}{p-2}$ for
$2<p\leq4$. In other words, for $4\leq p\leq\infty$ and $n\geq1$, there is a
constant $C_{\mathbb{K},2,p}^{\mathrm{pol}}\geq1$ (not depending on $n$) such
that
\[
\left(  {\sum\limits_{\left\vert \alpha\right\vert =2}}\left\vert a_{\alpha
}\right\vert ^{\frac{4p}{3p-4}}\right)  ^{\frac{3p-4}{4p}}\leq C_{\mathbb{K}%
,2,p}^{\mathrm{pol}}\left\Vert P\right\Vert ,
\]
for all $2$-homogeneous polynomials on $\ell_{p}^{n}$ given by $P(x_{1}%
,\ldots,x_{n})=\sum_{|\alpha|=m}a_{\alpha}\mathbf{{x}^{\alpha}}$. When
$2<p\leq4$ the optimal exponent $\frac{4p}{3p-4}$ is replaced by $\frac
{p}{p-2}$, which is also sharp$.$

When $m<p<2m$ the above inequality has a polynomial version due to Dimant and
Sevilla-Peris \cite{dimant}: given an $m$-homogeneous polynomial
$P(x_{1},\ldots,x_{n})=\sum_{|\alpha|=m}a_{\alpha}\mathbf{{x}^{\alpha}}$
defined on $\ell_{p}^{n}$ with $m<p<2m,$ there is a constant $C_{\mathbb{K}%
,m,p}^{\mathrm{pol}}\geq1$ (not depending on $n$) such that
\[
\left(  {\sum\limits_{\left\vert \alpha\right\vert =m}}\left\vert a_{\alpha
}\right\vert ^{\frac{p}{p-m}}\right)  ^{\frac{p-m}{p}}\leq C_{\mathbb{K}%
,m,p}^{\mathrm{pol}}\left\Vert P\right\Vert .
\]
Moreover the exponent $\frac{p}{p-m}$ is sharp. For $p\geq2m$, a similar
inequality replacing the optimal exponent $\frac{p}{p-m}$ by the optimal
exponent $\frac{2mp}{mp+p-2m}$ holds (this case is due to Praciano-Pereira
\cite{pra}).

In this note we extend the above inequality (keeping its sharpness) to the
case $p=m$ (we mention \cite{aapp} for a different approach for multilinear
forms; here, contrary to what happens in \cite{aapp}, we allow the left hand
side of the inequality to be the $\sup$ norm) . We also obtain the optimal
constant when we are restricted to $2$-homogeneous polynomials defined on
$\ell_{2}^{2}$ over the real scalar field. In a final appendix we show why the
optimal exponents of the bilinear Hardy--Littlewood inequality do not behave
smoothly (a similar argument holds for $m$-linear forms and $m$-homogeneous polynomials).

\section{The Hardy--Littlewood inequality for $m$-homogeneous polynomials on
$\ell_{m}$}

\bigskip Let us recall the $m$-linear Hardy--Littlewood inequalities:

\begin{itemize}
\item (Hardy--Littlewood/Praciano-Pereira \cite{hardy,pra}, 1934/1981) Let
$m\geq2$ be a positive integer and $p\geq2m.$ For all $m$--linear forms
$T:\ell_{p}^{n}\times\cdots\times\ell_{p}^{n}\rightarrow\mathbb{K}$ and all
positive integers $n$,
\begin{equation}
\textstyle\left(  \sum\limits_{j_{1},...,j_{m}=1}^{n}\left\vert T(e_{j_{1}%
},...,e_{j_{m}})\right\vert ^{\frac{2mp}{mp+p-2m}}\right)  ^{\frac
{mp+p-2m}{2mp}}\leq\left(  \sqrt{2}\right)  ^{m-1}\left\Vert T\right\Vert .
\label{yd}%
\end{equation}
Moreover, the exponent $2mp/(mp+p-2m)$ is optimal.
\end{itemize}

\bigskip

\begin{itemize}
\item (Hardy--Littlewood/Dimant--Sevilla-Peris \cite{hardy, dimant},
1934/2014) Let $m\geq2$ be a positive integer and $m<p<2m$. For all
$m$--linear forms $T:\ell_{p}^{n}\times\cdots\times\ell_{p}^{n}\rightarrow
\mathbb{K}$ and all positive integers $n$,
\begin{equation}
\textstyle\left(  \sum\limits_{j_{1},...,j_{m}=1}^{n}\left\vert T(e_{j_{1}%
},...,e_{j_{m}})\right\vert ^{\frac{p}{p-m}}\right)  ^{\frac{p-m}{p}}%
\leq\left(  \sqrt{2}\right)  ^{m-1}\left\Vert T\right\Vert . \label{0987}%
\end{equation}
Moreover, the exponent $p/(p-m)$ is optimal.
\end{itemize}

\bigskip From now on the optimal (and unknown) constants satisfying the above
inequalities are denoted by $C_{\mathbb{K},m,p}^{\mathrm{mult}}$.

We begin with the following lemma which is an adaptation of \cite[Proposition
2.2]{arq}. We present a proof for the sake of completeness:

\begin{lemma}
\label{pro:first_approach} If $P$ is an $m$-homogeneous polynomial of degree
$m$ on $\ell_{p}^{n},$ with $m<p<2m,$ given by $P(x_{1},\ldots,x_{n}%
)=\sum_{|\alpha|=m}a_{\alpha}\mathbf{{x}^{\alpha}}$, then
\[
\left(  {\sum\limits_{\left\vert \alpha\right\vert =m}}\left\vert a_{\alpha
}\right\vert ^{\frac{p}{p-m}}\right)  ^{\frac{p-m}{p}}\leq C_{\mathbb{K}%
,m,p}^{\mathrm{pol}}\left\Vert P\right\Vert
\]
with
\[
C_{\mathbb{K},m,p}^{\mathrm{pol}}\leq C_{\mathbb{K},m,p}^{\mathrm{mult}}%
\frac{m^{m}}{\left(  m!\right)  ^{\frac{p-m}{p}}}.
\]

\end{lemma}

\begin{proof}
Let $L$ be the symmetric $m$-linear form associated to $P$. From \cite{ddd4}
we have
\begin{align*}
{\sum\limits_{\left\vert \alpha\right\vert =m}}\left\vert a_{\alpha
}\right\vert ^{\frac{p}{p-m}}  &  ={\sum\limits_{\left\vert \alpha\right\vert
=m}}\left(  \binom{m}{\alpha}\left\vert L(e_{1}^{\alpha_{1}},\ldots
,e_{n}^{\alpha_{n}})\right\vert \right)  ^{\frac{p}{p-m}}\\
&  ={\sum\limits_{\left\vert \alpha\right\vert =m}}\binom{m}{\alpha}^{\frac
{p}{p-m}}\left\vert L(e_{1}^{\alpha_{1}},\ldots,e_{n}^{\alpha_{n}})\right\vert
^{\frac{p}{p-m}}.
\end{align*}
For all $\alpha$, the term $\left\vert L(e_{1}^{\alpha_{1}},\ldots
,e_{n}^{\alpha_{n}})\right\vert ^{\frac{p}{p-m}}$ appears $\binom{m}{\alpha}$
times in $\sum_{i_{1},\ldots,i_{m}=1}^{n}\left\vert L(e_{i_{1}},\ldots
,e_{i_{m}})\right\vert ^{\frac{p}{p-m}}$. Hence%
\[
{\sum\limits_{\left\vert \alpha\right\vert =m}}\binom{m}{\alpha}^{\frac
{p}{p-m}}\left\vert L(e_{1}^{\alpha_{1}},\ldots,e_{n}^{\alpha_{n}})\right\vert
^{\frac{p}{p-m}}={\sum\limits_{i_{1},\ldots,i_{m}=1}^{n}}\binom{m}{\alpha
}^{\frac{p}{p-m}}\frac{1}{\binom{m}{\alpha}}\left\vert L(e_{i_{1}}%
,\ldots,e_{i_{m}})\right\vert ^{\frac{p}{p-m}}%
\]

and, since $\binom{m}{\alpha}\leq m!$ we obtain
\[
{\sum\limits_{\left\vert \alpha\right\vert =m}}\binom{m}{\alpha}^{\frac
{p}{p-m}}\left\vert L(e_{1}^{\alpha_{1}},\ldots,e_{n}^{\alpha_{n}})\right\vert
^{\frac{p}{p-m}}\leq\left(  m!\right)  ^{\frac{p}{p-m}-1}{\sum\limits_{i_{1}%
,\ldots,i_{m}=1}^{n}}\left\vert L(e_{i_{1}},\ldots,e_{i_{m}})\right\vert
^{\frac{p}{p-m}}.
\]
We thus have, from the $m$-linear Hardy--Littlewood inequality,
\begin{align*}
\left(  {\sum\limits_{\left\vert \alpha\right\vert =m}}\left\vert a_{\alpha
}\right\vert ^{\frac{p}{p-m}}\right)  ^{\frac{p-m}{p}}  &  \leq\left(  \left(
m!\right)  ^{\frac{p}{p-m}-1}{\sum\limits_{i_{1},\ldots,i_{m}=1}^{n}%
}\left\vert L(e_{i_{1}},\ldots,e_{i_{m}})\right\vert ^{\frac{p}{p-m}}\right)
^{\frac{p-m}{p}}\\
&  =\left(  m!\right)  ^{1-\frac{p-m}{p}}\left(  {\sum\limits_{i_{1}%
,\ldots,i_{m}=1}^{n}}\left\vert L(e_{i_{1}},\ldots,e_{i_{m}})\right\vert
^{\frac{p}{p-m}}\right)  ^{\frac{p-m}{p}}\\
&  \leq\left(  m!\right)  ^{1-\frac{p-m}{p}}C_{\mathbb{K},m,p}^{\mathrm{mult}%
}\left\Vert L\right\Vert .
\end{align*}
On the other hand, it is well-known that
\[
\left\Vert L\right\Vert \leq\frac{m^{m}}{m!}\left\Vert P\right\Vert
\]
and hence
\[
\left(  {\sum\limits_{\left\vert \alpha\right\vert =m}}\left\vert a_{\alpha
}\right\vert ^{\frac{p}{p-m}}\right)  ^{\frac{p-m}{p}}\leq\left(  \left(
m!\right)  ^{1-\frac{p-m}{p}}\frac{m^{m}}{m!}\right)  C_{\mathbb{K}%
,m,p}^{\mathrm{mult}}\left\Vert P\right\Vert =\left(  \frac{m^{m}}{\left(
m!\right)  ^{\frac{p-m}{p}}}\right)  C_{\mathbb{K},m,p}^{\mathrm{mult}%
}\left\Vert P\right\Vert .
\]

\end{proof}

Now we are ready to state and prove our first result:

\begin{proposition}
[The Hardy--Littlewood inequality for $2$-homogeneous polynomials in $\ell
_{2}$]\label{9000}For all positive integers $n$ we have%
\[
\max_{|\alpha|=2}\left\vert a_{\alpha}\right\vert \leq4\sqrt{2}\left\Vert
P\right\Vert
\]
for all $P=\sum_{|\alpha|=2}a_{\alpha}\mathbf{{x}^{\alpha}}$ in $\mathcal{P}%
(^{2}\ell_{2}^{n})$. Moreover this result is optimal in the sense that the
$\sup$ norm in the left hand side cannot be replaced by any $\ell_{r}$-norm
without keeping the constant independent of $n$.
\end{proposition}

\begin{proof}
Let $2<p<4.$ It is well-known, from the Hardy--Littlewood inequality (see also
\cite{dimant}) for bilinear forms $T:\ell_{p}^{n}\times\ell_{p}^{n}%
\rightarrow\mathbb{K}$, that%
\begin{equation}
\left(
{\textstyle\sum\limits_{i,j=1}^{n}}
\left\vert T\left(  e_{i},e_{j}\right)  \right\vert ^{\frac{p}{p-2}}\right)
^{\frac{p-2}{p}}\leq\sqrt{2}\left\Vert T\right\Vert . \label{troca}%
\end{equation}
From the previous lemma we conclude that for all $Q=\sum_{|\alpha|=2}%
c_{\alpha}\mathbf{{x}^{\alpha}}$ in $\mathcal{P}(^{2}\ell_{p}^{n})$ we have
\[
\left(
{\textstyle\sum\limits_{|\alpha|=2}}
\left\vert c_{\alpha}\right\vert ^{\frac{p}{p-2}}\right)  ^{\frac{p-2}{p}}%
\leq2^{\frac{p+2}{p}}\sqrt{2}\left\Vert Q\right\Vert .
\]
Let $P=\sum_{|\alpha|=2}a_{\alpha}\mathbf{{x}^{\alpha}}$ be a polynomial in
$\mathcal{P}(^{2}\ell_{2}^{n})$. For all $p\in\left(  2,4\right)  $ let us
consider $P_{p}\in$ $\mathcal{P}(^{2}\ell_{p}^{n})$ given by the same rule as
$P$. We have%
\begin{align*}
\left(
{\textstyle\sum\limits_{|\alpha|=2}}
\left\vert a_{\alpha}\right\vert ^{\frac{p}{p-2}}\right)  ^{\frac{p-2}{p}}  &
\leq2^{\frac{p+2}{p}}\sqrt{2}\sup\left\{  \left\vert P_{p}\left(
x_{1},...,x_{n}\right)  \right\vert :%
{\textstyle\sum}
\left\vert x_{i}\right\vert ^{p}=1\right\} \\
&  =2^{\frac{p+2}{p}}\sqrt{2}\sup\left\{  \left\vert P\left(  x_{1}%
,...,x_{n}\right)  \right\vert :%
{\textstyle\sum}
\left\vert x_{i}\right\vert ^{p}=1\right\}  .
\end{align*}
Making $p\rightarrow2$ we obtain%
\[
\max_{|\alpha|=2}\left\vert a_{\alpha}\right\vert \leq4\sqrt{2}\left\Vert
P\right\Vert .
\]
Now we prove the optimality. Suppose that there is a $r<\infty$ and a constant
$C\geq1~$(not depending on $n$) such that%
\[
\left(
{\textstyle\sum\limits_{|\alpha|=2}}
\left\vert a_{\alpha}\right\vert ^{r}\right)  ^{\frac{1}{r}}\leq C\left\Vert
P\right\Vert
\]
for all $P=\sum_{|\alpha|=2}a_{\alpha}\mathbf{{x}^{\alpha}}$ in $\mathcal{P}%
(^{2}\ell_{2}^{n})$ and all $n.$ Let $p\in\left(  2,4\right)  $ be so that%
\[
r<\frac{p}{p-2}.
\]

Let $R=\sum_{|\alpha|=2}\beta_{\alpha}\mathbf{{x}^{\alpha}}$ be a polynomial
in $\mathcal{P}(^{2}\ell_{p}^{n})$ and let $R_{2}$ be the same polynomial, but
with domain $\ell_{2}^{n}.$ We thus have%
\begin{align*}
\left(
{\textstyle\sum\limits_{|\alpha|=2}}
\left\vert \beta_{\alpha}\right\vert ^{r}\right)  ^{\frac{1}{r}}  &  \leq
C\sup\left\{  \left\vert R_{2}\left(  x_{1},...,x_{n}\right)  \right\vert :%
{\textstyle\sum}
\left\vert x_{i}\right\vert ^{2}=1\right\} \\
&  =C\sup\left\{  \left\vert R\left(  x_{1},...,x_{n}\right)  \right\vert :%
{\textstyle\sum}
\left\vert x_{i}\right\vert ^{2}=1\right\} \\
&  \leq C\sup\left\{  \left\vert R\left(  x_{1},...,x_{n}\right)  \right\vert
:%
{\textstyle\sum}
\left\vert x_{i}\right\vert ^{p}=1\right\}
\end{align*}
for all $n$ and this is a contradiction in view of the optimality of the
exponent $\frac{p}{p-2}$ in the classical Hardy--Littlewood inequality.
\end{proof}

\bigskip

\begin{remark}
We recall the definition of the polarization constants for polynomials on
$\ell_{p}$ spaces:
\[
{\mathbb{K}}(m,p):=\inf\{M>0\,:\,\Vert L\Vert\leq M\Vert P\Vert\},
\]
where the infimum is taken over all $P\in{\mathcal{P}}(^{m}\ell_{p}^{n})$ and
$L$ is the unique symmetric $m$-linear form associated to $P$. As we have used
in Lemma \ref{pro:first_approach}, it is well known that in general
\[
\left\Vert L\right\Vert \leq\frac{m^{m}}{m!}\left\Vert P\right\Vert
\]
but for $\ell_{p}^{n}$ spaces the above estimate may be improved if we use
${\mathbb{K}}(m,p).$ For instance, a result due to Harris \cite{Harris}
asserts that
\[
{\mathbb{C}}(m,p)\leq\left(  \frac{m^{m}}{m!}\right)  ^{\frac{|p-2|}{p}},
\]
for all $p\geq1$, whenever $m$ is a power of $2$ (see also \cite{arq}). In
particular, if $m=2$ and $p>2$, we have, from the proof of the previous lemma,%
\begin{align*}
\left(  {\sum\limits_{\left\vert \alpha\right\vert =2}}\left\vert a_{\alpha
}\right\vert ^{\frac{p}{p-2}}\right)  ^{\frac{p-2}{p}}  &  \leq\left(
2^{1-\frac{p-2}{p}}\left(  \frac{2^{2}}{2!}\right)  ^{\frac{|p-2|}{p}}\right)
C_{\mathbb{C},2,p}^{\mathrm{mult}}\left\Vert P\right\Vert \\
&  =2\cdot C_{\mathbb{C},2,p}^{\mathrm{mult}}\left\Vert P\right\Vert
\end{align*}
for all $P$ on $\ell_{p}^{n}$, when working with complex scalars$.$
\end{remark}

\bigskip If we look for better constants we can isolate the case of complex
scalars of Proposition (\ref{9000}) and obtain the following (note that a
careful examination of \cite{dimant} shows that we can replace $\sqrt{2}$ by
$\frac{2}{\sqrt{\pi}}$ in (\ref{troca}) for the case of complex scalars):

\begin{proposition}
[The $2$-homogeneous Hardy--Littlewood inequality for $\ell_{2}$ and complex
scalars]\label{dois} For all $n\geq1$, we have%
\[
\max_{|\alpha|=2}\left\vert a_{\alpha}\right\vert \leq\frac{4}{\sqrt{\pi}%
}\left\Vert P\right\Vert
\]
for all $P=\sum_{|\alpha|=2}a_{\alpha}\mathbf{{x}^{\alpha}}$ in $\mathcal{P}%
(^{2}\ell_{2}^{n})$ over the complex scalar field. Moreover this result is
optimal in the sense of Theorem \ref{9000} .
\end{proposition}

A simple adaptation of the proof of Proposition \ref{9000} combined with the
$m$-linear version of the Hardy--Littlewood inequality due to Dimant and
Sevilla-Peris for $m<p<2m$ (see \cite[Proposition 4.1 (i)]{dimant}) gives us
the following general extension for the case $p=m:$

\begin{theorem}
[The Hardy--Littlewood inequality for $m$-homogeneous polynomials in $\ell
_{m}$]\bigskip Let $m\geq2$ be a positive integer. Given $n\geq1$, there is an
optimal constant $C_{\mathbb{K},m}\geq1$ (not depending on $n$) such that%
\[
\max_{|\alpha|=m}\left\vert a_{\alpha}\right\vert \leq C_{\mathbb{K}%
,m}\left\Vert P\right\Vert
\]
for all $P\in$ $\mathcal{P}(^{m}\ell_{m}^{n})$, with%
\[
C_{\mathbb{R},m}\leq\left(  \sqrt{2}\right)  ^{m-1}m^{m},
\]%
\[
C_{\mathbb{C},m}\leq\left(  \frac{2}{\sqrt{\pi}}\right)  ^{m-1}m^{m}.
\]
Moreover this result is optimal in the sense that the $\sup$ norm in the left
hand side cannot be replaced by any $\ell_{r}$-norm without keeping the
constant independent of $n$.
\end{theorem}

\begin{remark}
\bigskip If we look for better constants we can write the above estimates
depending on the polarization constants and we get%
\[
C_{\mathbb{R},m}\leq\left(  \sqrt{2}\right)  ^{m-1}\left(  m!\right)
{\mathbb{R}}(m,m)
\]%
\[
C_{\mathbb{C},m}\leq\left(  \frac{2}{\sqrt{\pi}}\right)  ^{m-1}\left(
m!\right)  {\mathbb{C}}(m,m).
\]

\end{remark}

\bigskip

\section{\bigskip The optimal constant for the case $m=2$ and $\ell_{2}^{2}$}

For all fixed $n\geq1$ let us define $C_{\mathbb{K}}$ $\left(  n\right)  $ as
the optimal constant satisfying
\[
\max_{|\alpha|=2}\left\vert a_{\alpha}\right\vert \leq C_{\mathbb{K}}\left(
n\right)  \left\Vert P\right\Vert
\]
for all $P:\ell_{2}^{n}\rightarrow\mathbb{K}$. It is simple to show that
$C_{\mathbb{R}}\left(  2\right)  \geq2.$ In fact, the $2$-homogeneous
polynomial%
\[
P_{2}:\ell_{2}^{2}\rightarrow\mathbb{R}%
\]
given by%
\[
P_{2}(x)=x_{1}x_{2}.
\]
has norm $1/2.$ From%
\[
\max_{|\alpha|=2}\left\vert a_{\alpha}\right\vert \leq C_{\mathbb{R}}\left(
2\right)  \left\Vert P_{2}\right\Vert
\]
we conclude that%
\[
C_{\mathbb{R}}\left(  2\right)  \geq2.
\]
In order to show that the optimal constant $C_{\mathbb{R}}\left(  2\right)  $
is precisely $2$ we will use the expression of the extremal polynomials on the
unit ball of $\mathcal{P}(^{2}\ell_{2}^{2}).$ The following result is due to
Choi and Kim (\cite{CHOI}):

\begin{theorem}
[Choi--Kim]For $p=2$, a $2$-homogeneous norm one polynomial $P(x,y)=ax^{2}%
+by^{2}+cxy$ is an extreme point of the unit ball of $\mathcal{P}(^{2}\ell
_{2}^{2})$ if, and only if,

(i) $\left\vert a\right\vert =\left\vert b\right\vert =1$, $c=0$ or

(ii) $a=-b$, $0<\left\vert c\right\vert \leq2$ and $4a^{2}=4-c^{2}.$
\end{theorem}

\bigskip From the Krein--Milman Theorem, we already know that the optimal
constants shall be searched within the extreme polynomials of the unit ball of
$\mathcal{P}(^{2}\ell_{2}^{2}).$ So we have:\bigskip

\begin{theorem}
For $\mathbb{K}=\mathbb{R}$, the optimal constant for the Hardy--Littlewood
inequality for $2$-homogeneous polynomials in $\mathcal{P}(^{2}\ell_{2}^{2})$
is $2$.
\end{theorem}

\begin{proof}
Let us denote by $C_{\mathbb{R}}\left(  2\right)  $ the optimal constant. For
all extremal polynomials given by the previous theorem we have
\[
\max\left\{  \left\vert a\right\vert ,\left\vert b\right\vert ,\left\vert
c\right\vert \right\}  \leq C_{\mathbb{R}}\left(  2\right)  \left\Vert
P\right\Vert =C_{\mathbb{R}}\left(  2\right)  .
\]
In the case (i) we have $C_{\mathbb{R}}\left(  2\right)  \geq1$ and in the
case (ii) we have%
\[
2=\max\left\{  \left\vert a\right\vert ,\sqrt{4-4a^{2}}:0<a<1\right\}  \leq
C_{\mathbb{R}}\left(  2\right)  \left\Vert P\right\Vert =C_{\mathbb{R}}\left(
2\right)  ,
\]
and thus the optimal constant $C_{\mathbb{R}}\left(  2\right)  $ is $2$.
\end{proof}

\begin{remark}
It was recently proved in \cite{1111} that, when $\mathbb{K}=\mathbb{R}$, the
optimal constants for the Hardy--Littlewood inequality for $2$-homogeneous
polynomials in $\mathcal{P}(^{2}\ell_{p}^{2})$ and $2<p<4$ is $2^{2/p}$ (the
case $p=4$ is proved in \cite{aaaa}). The above result shows that this formula
is also valid for our new version of the Hardy--Littlewood inequality for
$p=2,$ since $2^{2/2}=2.$
\end{remark}

\section{Appendix: why are the optimal exponents of the Hardy--Littlewood
inequality not smooth?}

\bigskip The original versions of the Hardy--Littlewood inequality for
bilinear forms can be stated as follows:

\begin{itemize}
\item \bigskip\cite[Theorems 2 and 4]{hardy} If $p,q\geq2$ are such that
\[
\frac{1}{2}<\frac{1}{p}+\frac{1}{q}<1
\]
then there is a constant $C\geq1$ such that%
\begin{equation}
\left(  \sum\limits_{j,k=1}^{\infty}\left\vert A(e_{j},e_{k})\right\vert
^{\frac{pq}{pq-p-q}}\right)  ^{\frac{pq-q-p}{pq}}\leq C\left\Vert A\right\Vert
\label{yh}%
\end{equation}
for all continuous bilinear forms $A:\ell_{p}\times\ell_{q}\rightarrow
\mathbb{R}$ (or $\mathbb{C}$). Moreover the exponent $\frac{pq}{pq-p-q}$ is optimal.

\item \bigskip\cite[Theorems 1 and 4]{hardy} If $p,q\geq2$ are such that
\[
\frac{1}{p}+\frac{1}{q}\leq\frac{1}{2}%
\]
then there is a constant $C\geq1$ such that%
\begin{equation}
\left(  \sum\limits_{j,k=1}^{\infty}\left\vert A(e_{j},e_{k})\right\vert
^{\frac{4pq}{3pq-2p-2q}}\right)  ^{\frac{3pq-2p-2q}{4pq}}\leq C\left\Vert
A\right\Vert \label{yh6}%
\end{equation}
for all continuous bilinear forms $A:\ell_{p}\times\ell_{q}\rightarrow
\mathbb{R}$ (or $\mathbb{C}$). Moreover the exponent $\frac{4pq}{3pq-2p-2q}$
is optimal.
\end{itemize}

\bigskip

Looking at both results, the natural question is: why does $\frac{1}{p}%
+\frac{1}{q}=\frac{1}{2}$ separate two different expressions for the optimal
exponents in (\ref{yh}) and (\ref{yh6})? In this appendix we revisite the
bilinear Hardy--Littlewood inequalities to justify this lack of smoothness. In
fact we show that, in a more precise sense (that will be clear soon in Remark
\ref{i999}) the exponent $\frac{pq}{pq-p-q}$ in (\ref{yh}) is not optimal. We
present a \textquotedblleft smooth\textquotedblright\ and optimal version
(Theorem \ref{hy09}) of the above Hardy--Littlewood theorems which,
surprisingly, is not entirely encompassed even by the ultimate very general
recent extensions of the Hardy--Littlewood inequalities (as those from
\cite{n}).

We begin by recalling a general version of the Kahane--Salem--Zygmund
inequality, which appears in \cite[Lemma 6.2]{alb}.

\begin{lemma}
[Kahane--Salem--Zygmund inequality (extended form)]\label{LEMKSZ} Let
$m,N\geq1$, $p_{1},\dots,p_{m}\in\lbrack1,\infty]$ and let, for $p\geq1$,
\[
\alpha(p)=\left\{
\begin{array}
[c]{ll}%
\displaystyle\frac{1}{2}-\frac{1}{p} & \text{ if }p\geq2\\
0 & \text{ otherwise.}%
\end{array}
\right.
\]
Then there exists a $m$-linear map $A:\ell_{p_{1}}^{N}\times\dots\times
\ell_{p_{m}}^{N}\rightarrow\mathbb{K}$ of the form
\[
A(z^{(1)},\dots,z^{(m)})=\sum_{i_{1},\dots,i_{m}}\pm z_{i_{1}}^{(1)}\cdots
z_{i_{m}}^{(m)}%
\]
such that
\[
\Vert A\Vert\leq C_{m}N^{\frac{1}{2}+\alpha(p_{1})+\dots+\alpha(p_{m})}%
\]
for some constant $C_{m}>0.$
\end{lemma}

If we look at \bigskip\cite[Theorem 2]{hardy} we can realize (see \cite[page
247]{hardy}) that in fact the authors prove that, for $\frac{1}{2}<\frac{1}%
{p}+\frac{1}{q}<1$, there is a constant $C\geq1$ such that%
\begin{equation}
\left(  \sum\limits_{k=1}^{\infty}\left(  \sum\limits_{j=1}^{\infty}\left\vert
A(e_{j},e_{k})\right\vert ^{2}\right)  ^{\frac{\lambda}{2}}\right)  ^{\frac
{1}{\lambda}}\leq C\left\Vert A\right\Vert \label{dan}%
\end{equation}
with $\lambda=\frac{pq}{pq-p-q},$ for all continuous bilinear forms
$A:\ell_{p}\times\ell_{q}\rightarrow\mathbb{R}$ (or $\mathbb{C}$)$.$ Since in
this case we have $2<\lambda$, the authors use a trivial estimate to conclude,
from (\ref{dan}), that
\begin{equation}
\left(  \sum\limits_{j,k=1}^{\infty}\left\vert A(e_{j},e_{k})\right\vert
^{\lambda}\right)  ^{\frac{1}{\lambda}}\leq C\left\Vert A\right\Vert ,
\label{uuuiii}%
\end{equation}
with $\lambda=\frac{pq}{pq-p-q},$ for all continuous bilinear forms
$A:\ell_{p}\times\ell_{q}\rightarrow\mathbb{R}$ (or $\mathbb{C}$)$.$ The proof
that the exponent $\frac{pq}{pq-p-q}$ in (\ref{uuuiii}) is sharp is quite
simple (we now use an idea taken from \cite{dimant}) and stress that the usual
approach via Kahane--Salem--Zygmund inequality is not effective here (why? due
to the \textquotedblleft rough\textquotedblright\ estimation when passing from
(\ref{dan}) to (\ref{uuuiii})!). To prove the optimality, it suffices to
consider the bilinear form $A_{n}:\ell_{p}\times\ell_{q}\rightarrow\mathbb{R}$
(or $\mathbb{C}$) given by $A_{n}(x,y)=\sum\limits_{j=1}^{n}x_{j}y_{j}$ and
use H\"{o}lder's inequality. In fact, since $\frac{1}{p}+\frac{1}{q}+\frac
{1}{\lambda}=1$, we have%

\begin{equation}
\left\Vert A_{n}\right\Vert \leq n^{1/\lambda}. \label{mas2}%
\end{equation}
If (\ref{uuuiii}) would hold for a certain $r$ instead of $\lambda$, combining
with (\ref{mas2}) we would obtain%
\[
n^{\frac{1}{r}}\leq Cn^{\frac{1}{\lambda}}%
\]
for all $n$, and thus%
\[
r\geq\lambda=\frac{pq}{pq-p-q}.
\]

As a matter of fact, even if we consider sums in just one index (i.e., $j=k$),
the exponent $\frac{pq}{pq-p-q}$ in (\ref{uuuiii}) is still optimal (observe
that $A_{n}$ is a kind of diagonal form). However, what does it exactly mean
that $\frac{pq}{pq-p-q}$ is optimal in (\ref{uuuiii}) in the usual sense? It
means (also in the sense of \cite{hardy}) that for both indexes $j,k$ we can
not take simultaneously exponents smaller than $\frac{pq}{pq-p-q}$. In other
words, re-writing (\ref{uuuiii}) as
\[
\left(  \sum\limits_{j=1}^{\infty}\left(  \left(  \sum\limits_{k=1}^{\infty
}\left\vert A(e_{j},e_{k})\right\vert ^{r}\right)  ^{\frac{1}{r}}\right)
^{s}\right)  ^{\frac{1}{s}}\leq C\left\Vert A\right\Vert
\]
we cannot have $r=s<\frac{pq}{pq-p-q}.$ But a different question, motivated by
(\ref{dan}), would be: is it possible to have $\left(  r,s\right)  $
satisfying the above inequality with $r=2$ and $s<\frac{pq}{pq-p-q}$ or with
$r<2$ and $s=\frac{pq}{pq-p-q}?$ So, a new question arises: Is (\ref{dan})
sharp in this sense? We stress that the trivial fact that
\[
\left(  \sum\limits_{j,k=1}^{\infty}\left\vert A(e_{j},e_{k})\right\vert
^{\frac{p}{p-2}}\right)  ^{\frac{p-2}{p}}\leq\left(  \sum\limits_{k=1}%
^{\infty}\left(  \sum\limits_{j=1}^{\infty}\left\vert A(e_{j},e_{k}%
)\right\vert ^{2}\right)  ^{\frac{\lambda}{2}}\right)  ^{\frac{1}{\lambda}}%
\]
plus the fact that the exponent $\lambda=\frac{pq}{pq-p-q}$ is sharp in
(\ref{uuuiii}) in the sense of \cite{hardy} does not assure that the exponents
$\lambda$ or $2$ in (\ref{dan}) are sharp in our sense: for this task we need
the Kahane--Salem--Zygmund inequality. In fact, if%
\[
\left(  \sum\limits_{j=1}^{\infty}\left(  \left(  \sum\limits_{k=1}^{\infty
}\left\vert A(e_{j},e_{k})\right\vert ^{2}\right)  ^{\frac{1}{2}}\right)
^{s}\right)  ^{\frac{1}{s}}\leq C\left\Vert A\right\Vert ,
\]
using the bilinear form $A$ from the Kahane--Salem--Zygmund inequality we
have, for all $N$,%
\[
\left(  N\cdot\left(  N^{\frac{1}{2}}\right)  ^{s}\right)  ^{\frac{1}{s}}\leq
CN^{\frac{1}{2}+\left(  \frac{1}{2}-\frac{1}{p}\right)  +\left(  \frac{1}%
{2}-\frac{1}{q}\right)  },
\]
and thus%
\[
N^{\frac{1}{2}+\frac{1}{s}}\leq CN^{\frac{3}{2}-\frac{1}{p}-\frac{1}{q}},
\]
i.e.,%
\[
s\geq\lambda.
\]
In the case in which the exponent $\frac{pq}{pq-p-q}$ is untouched we show
that the exponent $2$ can not be improved using a similar argument.

\begin{remark}
\label{i999}We note that in our \textquotedblleft more
precise\textquotedblright\ sense of optimality, the exponent $\frac
{pq}{pq-p-q}$ in (\ref{uuuiii}) is not optimal, because the \textquotedblleft
first\textquotedblright\ exponent $\lambda=\frac{pq}{pq-p-q}$ can be improved
to $2.$
\end{remark}

Now, if we turn our attention to \cite[Theorem 1]{hardy} we can also realize
that from \cite{hardy} we also have
\begin{equation}
\left(  \sum\limits_{k=1}^{\infty}\left(  \sum\limits_{j=1}^{\infty}\left\vert
A(e_{j},e_{k})\right\vert ^{2}\right)  ^{\frac{\lambda}{2}}\right)  ^{\frac
{1}{\lambda}}\leq C\left\Vert A\right\Vert \label{tre}%
\end{equation}
for $\lambda=\frac{pq}{pq-p-q}$ (this fact is also observed in \cite[Theorem
1]{tonge}, however with no mention to its eventual optimality, and for the
case of complex scalars). Again, a simple consequence of the
Kahane--Salem--Zygmund inequality asserts that the exponents of (\ref{tre})
are sharp (in the sense that $\lambda$ can not be improved keeping the
exponent $2$ as it is and vice-versa); we left the details for the reader. So,
from (\ref{dan}) and (\ref{tre}) we can rewrite, in a unifying and optimal
form, the results of Hardy and Littlewood as follows:

\begin{theorem}
[Hardy and Littlewood - revisited]\label{hy09}If $p,q\geq2$ and $\frac{1}%
{p}+\frac{1}{q}<1$, then there is a constant $C\geq1$ such that%
\begin{equation}
\left(  \sum\limits_{k=1}^{\infty}\left(  \sum\limits_{j=1}^{\infty}\left\vert
A(e_{j},e_{k})\right\vert ^{2}\right)  ^{\frac{\lambda}{2}}\right)  ^{\frac
{1}{\lambda}}\leq C\left\Vert A\right\Vert \label{i99p}%
\end{equation}
for $\lambda=\frac{pq}{pq-p-q},$ for all continuous bilinear forms $A:\ell
_{p}\times\ell_{q}\rightarrow\mathbb{R}$ (or $\mathbb{C}$). We, as usual,
consider $c_{0}$ instead of $\ell_{\infty}$ when $p=\infty.$ The exponents are
optimal in the sense that $\lambda$ can not be improved keeping the exponent
$2$ nor the exponent $2$ can be improved keeping the exponent $\lambda$.
\end{theorem}

\begin{remark}
For $\frac{1}{2}<\frac{1}{p}+\frac{1}{q}<1$ we have $2<\lambda$ and we can not
use a Minkowski's type result as in \cite[Section 3]{alb} to interchange the
positions of $2$ and $\lambda$. For this reason, even making a
\textquotedblleft rough\textquotedblright\ approximation when replacing $2$ by
$\lambda$ when passing from (\ref{dan}) to (\ref{uuuiii}), the resulting
estimate (\ref{uuuiii}) is still sharp in the usual sense, as we mentioned
before. We stress that it is in fact impossible in this case to interchange
$2$ and $\lambda$ (even looking for stronger arguments than a Minkowski's type
result). The reason is quite simple: if this was possible, by
\textquotedblleft interpolating\textquotedblright\ the resulting exponents
$\left(  2,\lambda\right)  $ and $\left(  \lambda,2\right)  $ with
$\theta=1/2$ in the sense of \cite[Section 2]{alb} we would obtain an
improvement of (\ref{uuuiii}) (in the usual sense, i.e., a smaller exponent
would be valid for all indexes), and we know that this is not possible.
\end{remark}

\begin{remark}
The fact that $2$ and $\lambda$ can not be interchanged when $\frac{1}%
{2}<\frac{1}{p}+\frac{1}{q}<1$ is certainly the reason of the absence of the
full content of Theorem \ref{hy09} in the very general paper \cite{n} (see
\cite[Theorem 1.1]{n}).
\end{remark}

For $\frac{1}{p}+\frac{1}{q}\geq\frac{1}{2}$ we have $\lambda<2$, and since
there is an obvious symmetry between $j$ and $k$, a consequence of Minkowski's
inequality allows us (as in \cite[Section 2]{alb}) to interchange the
positions of $2$ and $\lambda$ and obtain%
\[
\left(  \sum\limits_{k=1}^{\infty}\left(  \sum\limits_{j=1}^{\infty}\left\vert
A(e_{j},e_{k})\right\vert ^{\lambda}\right)  ^{\frac{2}{\lambda}}\right)
^{\frac{1}{2}}\leq C\left\Vert A\right\Vert .
\]
Now \textquotedblleft interpolating the multiple exponents\textquotedblright%
\ $\left(  \lambda,2\right)  $ and $\left(  2,\lambda\right)  $ with
$\theta=1/2$ in the sense of \cite[Section 2]{alb}$,$ or using the
H\"{o}lder's inequality for mixed sums (see \cite{benede}), we obtain
(\ref{yh6}) as a corollary. In fact, varying the weight $\theta$ from $0$ to
$1$, we recover a family of optimal inequalities as in \cite{alb, n}.

\end{document}